\RequirePackage{fix-cm}
\documentclass[smallextended]{svjour3}       
\smartqed  
\usepackage{graphicx}
\usepackage{amssymb,amsmath,amsfonts,color,url}
\usepackage{rotating}

\usepackage{multirow}
\usepackage{xifthen}
\newtheorem{thm}{Theorem}[section]

\newtheorem{lem}[thm]{Lemma} 
\newtheorem{ex}[thm]{Example}
\newtheorem{conj}[thm]{Conjecture}

\numberwithin{equation}{section}

\def \N { {\mathbb N} }
\def \Q { {\mathbb Q} }
\def \R { {\mathbb R} }
\def \C { {\mathbb C} }
\def \Z { {\mathbb Z} }
\def \cP { { P} }
\def \F { {\mathcal F} }

\def \w { {\bf w} }

\def \V { {\mathbb V} }

\def \MS {{\mathcal P^*}}

\def \wordsnl { \Omega_{S,T}}
\def \wordsnlt { \Omega_{T}}

\def\ve#1{\mathchoice{\mbox{\boldmath$\displaystyle\bf#1$}}
{\mbox{\boldmath$\textstyle\bf#1$}}
{\mbox{\boldmath$\scriptstyle\bf#1$}}
{\mbox{\boldmath$\scriptscriptstyle\bf#1$}}}

 \DeclareMathOperator{\conv}{conv}
 \DeclareMathOperator{\vertices}{vert}
 \DeclareMathOperator{\cone}{cone}
 
 \newcommand\Side[1]{\begin{sideways}{\small #1}\end{sideways}}

\definecolor{darkgreen}{rgb}{0,0.6,0} 
\definecolor{darkyellow}{rgb}{0.8,0.6,0} 

\begin{document}

\title{Markov degree of the three-state toric homogeneous Markov chain model
}


\author{David Haws \and
Abraham Mart\'in del Campo \thanks{Research of Mart\'in del Campo
  supported in part by NSF grant 
  DMS-915211} \and
Akimichi Takemura \and
Ruriko Yoshida
}

\authorrunning{Haws \and Mart\'in del Campo \and Takemura \and Yoshida} 

\institute{D. Haws \at IBM, Watson in Yorktown Heights, New York USA\\
TEL:+1-914-945-2738\\
\email{dchaws@gmail.com} \and  
A. Mart\'in del
  Campo \at IST Austria, Am Campus 1, 
A - 3400, Klosterneuburg, 
Austria,\\  
     TEL:+43-(0)2243-9000\\
    \email{abraham.mc@ist.ac.at}  
    \and A.
  Takemura \at  University of Tokyo, Bunkyo, Tokyo 113-0033 Tokyo Japan\\
   TEL:+81-(0)3-5841-6940\\ \email{takemura@stat.t.u-tokyo.ac.jp} \and 
  R.   Yoshida \at University of Kentucky, 725 Rose Street Lexington KY
    40536-0082  USA\\ 
TEL:+1-859-257-5698  \\
\email{ruriko.yoshida@uky.edu}
}

\date{Received: date / Accepted: date}

\maketitle

\begin{abstract}
We consider the three-state toric homogeneous Markov chain model (THMC) without loops and initial parameters. At time $T$, the size of the design
matrix is $6 \times 3\cdot 2^{T-1}$ and the convex hull of its columns is the model polytope.
We study the behavior
of this polytope
for $T\geq 3$  and we show that
it is defined by $24$ facets for all $T\ge 5$. 
Moreover, we give a complete description of these facets.
From this, we deduce that the toric ideal associated with the design
matrix is generated by binomials of degree at most $6$. 
Our proof is based on a result due to
Sturmfels, who gave a 
bound on the degree of the generators of a 
toric ideal, provided the normality of the corresponding toric
variety. In our setting, we established the normality of the toric variety associated to the THMC model by studying the geometric properties of
the model polytope. 
\keywords{Toric ideals \and toric homogeneous Markov chains \and
  polyhedron \and
  semigroups}
\end{abstract}

\section{Introduction}
\label{intro}
A discrete time Markov chain, $X_t$ for $t = 1, 2, \ldots$,  is a stochastic
process with the Markov property, that is  $P(X_{t+1} =
y|X_1=x_1,\ldots,X_{t-1}=x_{t-1},X_t = x)   
= P(X_{t+1} = y | X_t = x)$ for any states $x, \, y$.  
Discrete time Markov chains have applications in several fields, such as  
physics, chemistry, information sciences, economics, finances,
mathematical biology, social sciences, and statistics \cite{stochastics}. 
In this paper, we consider a discrete time Markov chain $X_t$ over a set of states $[S] = \{ 1, \ldots, S\}$, with
$t=1,\ldots, T$ ($T\geq 3$), focusing on the case $S=3$.

Discrete time Markov chains are often used in statistical models to
fit the observed data from a random physical process. Sometimes, in
order to simplify the model, it is convenient to consider
time-homogeneous Markov chains, where the transition probabilities do
not depend on the time, in other words, when 
\[
P(X_{t+1} {=} y| X_{t}{ = }x) = P(X_2 {=} y| X_1 {=} x) 
\quad \forall \, x, \, y \in [S] \text{ and for any }t =1\ldots, T{-}1.
\]

Let $\w = s_1\cdots s_T$ denote a word of
length $T$ on states $[S]$. Let $p(\w)$ denote the likelihood of observing the word $\w$. In the time-homogeneous Markov chain model, this likelihood is written as the product of probabilities 
\begin{equation}\label{thmc}
P(\w) = \pi_{s_1}p_{s_1,s_2}\cdots p_{s_{T-1},s_T},
\end{equation}
where, $\pi_{s_i}$ indicates the initial distribution 
at the first state, and $p_{s_i, s_j}$ are the transition probabilities
from state $s_i$ to $s_j$. 
In the usual time-homogeneous Markov chain model it is assumed that the row sums of the
transition probabilities are equal to one: 
$\sum_{j=1}^S p_{i,j}=1$, $\forall i\in [S]$.
In addition, the toric homogeneous Markov chain (THMC) model is also described by \eqref{thmc}, but where the parameters $p_{i,j}$ are free and the row sums of the transition probabilities are not restricted.

In many cases the parameters $\pi_{s_1}$  for the initial distribution
are known, or sometimes these parameters are all constant, namely
$\pi_{1}=\pi_{2}=\cdots=\pi_{S}=c$; in this situation it is
no longer necessary to  
take them in consideration in expression~\eqref{thmc}, making the model simpler.
Another simplification that arises from
practice is when the only transition probabilities considered are those between
two different states, i.e.\ when $p_{i,j}=0$ whenever $i=j$; this
situation is referred  as the THMC model without self-loops. In this paper, we consider both simplifications of the THMC model.

In order for a statistical model to reflect the observed data, it has to pass
a goodness-of-fit test. 
For instance, for the time-homogeneous Markov chain model, it is
necessary to test if the assumption of
time-homogeneity fits the observed data.
In 1998, Diaconis-Sturmfels developed a Markov Chain Monte Carlo method (MCMC) for goodness-of-fit test by using \emph{Markov bases}~\cite{diaconis-sturmfels}.

 A Markov basis is a set of moves
 between objects with the same sufficient statistics
 so that the transition graph for the MCMC 
 is guaranteed to
 be connected for any observed value of the sufficient statistics (see Section \ref{MB} and \cite{stochastics}). 
In algebraic terms, a Markov basis is a generating set of a \emph{toric ideal} defined as the kernel of a monomial map between two polynomial rings. In algebraic statistics, the monomial map comes from the \emph{design matrix} associated with a statistical model.

In \cite{Hara:2010vn}, the authors provided a full description of the Markov bases
for the THMC model in two states (i.e.\ when $S=2$) which does not
depend on $T$, even though the toric ideal lies on 
a polynomial ring with 
$2^T$ indeterminates. Inspired by their work, we study the   
%
algebraic and polyhedral properties of the Markov bases of the
three-state THMC model without  
initial parameters and without self-loops. 
We showed that for arbitrarily large time $T \geq 5$, the {\em model 
  polytope} --the convex hull of the columns of the design 
matrix-- has only $24$ facets and we
provide a complete description of them.  Moreover, by showing the normality of
the polytope, we deduced that the Markov bases of the model consist of
binomials of degree at most $6$. 

The outline of this paper is as follows. In Section \ref{notation}, we
recall some definitions from Markov bases theory. In Section
\ref{facets}, we explicitly describe the hyperplane representation of
the model polytope for the 
three-state THMC model without self-loops for any time $T \geq 5$.  In Section
\ref{normal}, we show that the model polytope is normal for arbitrary $T \geq 3$, this is equivalent to show  that the semigroup generated by the
columns of the design matrix is integrally closed.  Finally, using these 
results, we prove the bound on the degree of the Markov
bases in Section \ref{discussion}; and we conclude that section with
some observations based on the analysis of our computational
experiments.

\section{Notation}\label{notation}
Let $\wordsnl$ be the set of all words of length $T$ on states $[S]$ such that
every word has no self-loops; that is, if $\ve w=(s_1,\ldots,s_T) \in \wordsnl$
then $s_i \neq s_{i+1}$ for $i = 1,\ldots,T{-}1$. We define $\MS(\wordsnl)$ to be
the set of all multisets of words in $\wordsnl$.

Let $\V$ be the real vector space with basis $\wordsnl$
and note that $\V \cong \R^{S(S-1)^{T-1}}$.  We recall some
definitions from the book of Pachter and Sturmfels \cite{Pachter:2005kx}.  Let
$A = (a_{ij})$ be a non-negative integer $d{\times} m$-matrix with the property
that all column sums are equal:
\begin{equation*}
\sum_{i=1}^d a_{i1} = \sum_{i=1}^d a_{i2} = \cdots = \sum_{i=1}^d a_{im}. 
\end{equation*}
Write $A = [\ve a_1 \; \ve a_2 \; \cdots \; \ve a_m]$ where $\ve a_j$ are the column vectors of
$A$ and define $\theta^{\ve a_j} = \prod_{i=1}^d \theta_i^{a_{ij}}$ for $j=
1,\ldots,m$. The \emph{toric model} of $A$ is the image of the orthant
$\R^d_{\geq 0}$ under the map 
\begin{equation*}
f: \R^d \rightarrow \R^m, \quad \theta \mapsto \frac{1}{\sum_{j=1}^m \theta^{\ve a_j}}\left( \theta^{\ve a_1}, \ldots, \theta^{\ve a_m} \right).
\end{equation*}
Here we have $d$ parameters $\theta = (\theta_1,\ldots,\theta_d)$ and a
discrete state space of size $m$. In our setting, the discrete space will be
the set of all possible words on $[S]$ of length $T$ without self-loops ($\wordsnl$) and we
can think of $\theta_1,\ldots, \theta_d$ as the probabilities
$p_{1,2}, p_{1,3}, \dots, p_{S-1,S}$. 

In this paper, we focus on the THMC model without initial parameters and with no self-loops in three states, (i.e., $S=3$), which is parametrized by
%
$6$ positive real variables:  $p_{1,2},\, p_{1,3},\, p_{2,1}, \, p_{2,3}$,  
$p_{3,1},\, p_{3,2}$.
In this case, we only write $\wordsnlt$ instead of $\Omega_{3,T}$. The number of parameters is $d=6$ and the size of the discrete space
is $m = 3\cdot 2^{T-1}$, which is precisely the number of words in $\wordsnlt$.
The model we study is thus the toric model represented by the $6 \times
3\cdot 2^{T-1}$ matrix $A^{T}$, which will be referred to as the \emph{design matrix} for the model on $3$ states
with time $T$.
 The rows of $A^T$ are indexed by elements in $\Omega_2$ and the columns are indexed by words
in $\wordsnlt$.  The entry of $A^{T}$ indexed by row $\sigma_1\sigma_2 \in
\Omega_2$, and column $\ve w=(s_1,\ldots,s_T) \in
\wordsnlt$ is equal to the cardinality of the set $\left\{\, i \in
\{1,\ldots,T{-}1\} \mid \sigma_1 \sigma_2 = s_i s_{i+1} \, \right\}$.
\begin{ex}
Ordering $\Omega_2$ and
$\wordsnlt$ lexicographically, and letting $T=4$, the matrix
$A^{4}$ is:
\begin{center}
\setlength{\tabcolsep}{4pt}
\begin{tabular}{c|cccccccccccccccccccccccc}
& \Side{\scriptsize 1212 } & \Side{\scriptsize 1213 } & \Side{\scriptsize 1231 } & \Side{\scriptsize 1232 } & \Side{\scriptsize 1312 } & \Side{\scriptsize 1313 } & \Side{\scriptsize 1321 } & \Side{\scriptsize 1323 } & \Side{\scriptsize 2121 } & \Side{\scriptsize 2123 } & \Side{\scriptsize 2131 } & \Side{\scriptsize 2132 } & \Side{\scriptsize 2312 } & \Side{\scriptsize 2313 } & \Side{\scriptsize 2321 } & \Side{\scriptsize 2323 } & \Side{\scriptsize 3121 } & \Side{\scriptsize 3123 } & \Side{\scriptsize 3131 } & \Side{\scriptsize 3132 } & \Side{\scriptsize 3212 } & \Side{\scriptsize 3213 } & \Side{\scriptsize 3231 } & \Side{\scriptsize 3232 } \\
\hline
12 & 2 & 1 & 1 & 1 & 1 & 0 & 0 & 0 & 1 & 1 & 0 & 0 & 1 & 0 & 0 & 0 & 1 & 1 & 0 & 0 & 1 & 0 & 0 & 0  \\
13 & 0 & 1 & 0 & 0 & 1 & 2 & 1 & 1 & 0 & 0 & 1 & 1 & 0 & 1 & 0 & 0 & 0 & 0 & 1 & 1 & 0 & 1 & 0 & 0  \\
21 & 1 & 1 & 0 & 0 & 0 & 0 & 1 & 0 & 2 & 1 & 1 & 1 & 0 & 0 & 1 & 0 & 1 & 0 & 0 & 0 & 1 & 1 & 0 & 0  \\
23 & 0 & 0 & 1 & 1 & 0 & 0 & 0 & 1 & 0 & 1 & 0 & 0 & 1 & 1 & 1 & 2 & 0 & 1 & 0 & 0 & 0 & 0 & 1 & 1  \\
31 & 0 & 0 & 1 & 0 & 1 & 1 & 0 & 0 & 0 & 0 & 1 & 0 & 1 & 1 & 0 & 0 & 1 & 1 & 2 & 1 & 0 & 0 & 1 & 0  \\
32 & 0 & 0 & 0 & 1 & 0 & 0 & 1 & 1 & 0 & 0 & 0 & 1 & 0 & 0 & 1 & 1 & 0 & 0 & 0 & 1 & 1 & 1 & 1 & 2 
\end{tabular}
\end{center} 
\end{ex}
\subsection{Sufficient statistics, ideals, and Markov basis}\label{MB}

Let $A^T$ be the design matrix for the THMC model without initial parameters and with
no self-loops. The column of $ A^{T}$ indexed by $\ve w \in \wordsnlt$ is denoted by $
\ve a_{\ve w}^T$. 
Thus, by
extending linearly, the map $ A^{T}: \V \rightarrow \R^{6}$ is
well-defined.

Let $W = \{ w_1,\ldots,w_N \} \in \MS(\wordsnlt)$ where we regard $W$ as
observed data which can be summarized in the \emph{data vector} $\ve u \in
\N^{3\cdot 2^{T-1}}$, where $\N=\{0,1,\dots\}$. We index $\ve u$ by words in $\wordsnlt$, so the coordinate representing for the word $\ve w$ in the
vector $\ve u$ is denoted by $u_{\ve w}$, and its value is the number of words in $W$ equal to $\ve w$. Note since $ A^{T}$ is linear
then $ A^{T}\ve u$ is well-defined. 
For $W$ from $\MS(\wordsnlt)$, let $\ve u$ be its data vector, the
\emph{sufficient statistics} for the model are stored in the vector $A^{T} \ve
u$.  Often the data vector $\ve u$ is also referred to as a \emph{contingency
table}, in which case $A^T \ve u$ is referred to as the \emph{marginals}. 

The design matrix $A^T$ above defines a toric ideal which is of central interest
in this paper, as their sets of generators are in bijection with the Markov
bases. The toric ideal $I_{A^T}$ is defined as the kernel of the homomorphism of
polynomial rings $\psi:\C[\{\,P(\w) \mid \w \in \wordsnl\,\}] \rightarrow
\C[\{\,p_{ij} \mid i,j \in [3],\, i \neq j\,\}]$ defined by
$\psi(P(\w))=p_{s_1,s_2}\cdots p_{s_{T-1},s_T}$, where  
$\{\,P(\w) \mid \w \in \wordsnl\,\}$ is regarded as a set of
indeterminates. 

Let $\ve b\in \N^{d}$ be a set of marginals. The set of contingency tables
with marginals $\ve b$
is called a \emph{fiber} which we denote
by $\F_{\ve b} = \{\,\ve x \in \N^{m} \mid A^T \ve x = \ve b \,\}$.
A \emph{move} $\ve z \in \Z^{m}$ is an integer vector satisfying $A^T \ve z =
0$.  A {\em Markov basis} for our model 
is 
a finite set $\mathcal Z$ of moves satisfying that for all $\ve b$
and all pairs $\ve x, \ve y \in \F_{\ve b}$ there exists a sequence $\ve
z_1,\ldots,\ve z_K \in \mathcal Z$ such that

\begin{equation*}
\ve y = \ve x + \sum_{k=1}^K \ve z_k, \quad \text{with }\ve x + \sum_{k=1}^l \ve
z_k \geq \ve 0, \; \text{ for all } l=1,\ldots,K.
\end{equation*}
A
{\em minimal Markov basis} is a Markov basis which is minimal in terms of
inclusion.  See Diaconis and Sturmfels~\cite{diaconis-sturmfels} for more
details on Markov bases and their toric ideals.  

\subsection{State Graph}

We give here a useful tool to visualize multisets of $\MS(\wordsnlt)$.
Given any multiset $W \in \MS(\wordsnlt)$ we consider the directed multigraph
called the \emph{state graph} $G(W)$. The vertices of $G(W)$ are given by the
three states $\{1,2,3\}$ and the directed edges $i \to j$ are given by the transitions from
state $i$ to $j$ in $\ve w \in W$. Thus, we regard $\ve w \in W$ as a path with 
$T{-}1$ edges (steps, transitions) in $G(W)$. 

We illustrate the state graph $G(W)$ of the
multiset $W = \{(12132),(12321)\}$ of paths with length $4$ in Figure \ref{stateexone}. Notice that the state graph in this figure is also the 
the state graph for the multiset $\overline{W} = \{(13212),(21232)\}$.


\begin{figure}[!htp]
\begin{center}
\scalebox{1}{
\includegraphics{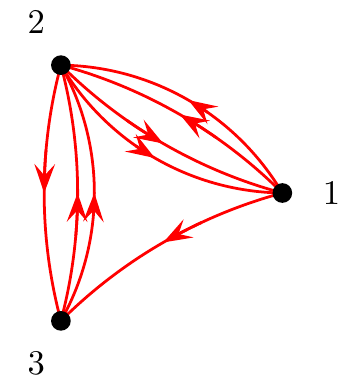}
}
\end{center}
\caption{The state graph of $W = \{(12132),(12321)\}$ and $\overline{W} = \{(13212),(21232)\}$.}
\label{stateexone}
\end{figure}

From the definition of state graph it is clear that it records the transitions
in a given multiset of words and we state the following proposition.
\begin{proposition}[Proposition 2.1 in \cite{Haws:2011fk}] 
\label{fiberstateequiv}
Let $A^T$ be the design matrix for the THMC, and $W,\overline W \in \MS(\wordsnl)$. Then $A^T (W)  = A^T (\overline W)$ if and only if $G(W) = G(\overline W)$. 
\end{proposition}

Throughout this paper we alternate between terminology of the multiset of words $W$ and
the graph $G(W)$ it defines.

\section{Facets of the design polytope}\label{facets}

\subsection{Polytopes}
We recall some necessary definitions from polyhedral geometry and we refer the
reader to the book of Schrijver~\cite{Schrijver1986Theory-of-linea} for more
details. The \emph{convex hull} of $\{\ve a_1,\ldots,\ve a_m\} \subset \R^d$ is
defined as 
\begin{equation*}   
\conv (\ve a_1,\ldots,\ve a_m) := \left\{\, \ve x \in \R^d \mid \ve x = \sum_{i=1}^m \lambda_i \ve a_i, \, \, \sum_{i=1}^m \lambda_i = 1, \, \lambda_i \geq 0 \, \right\}.  
\end{equation*} 

A \emph{polytope} $\cP$ is the convex hull of finitely many points. We say $F
\subseteq \cP$ is a {\em face} of the polytope $\cP$ if there exists a vector
$\ve c$ such that $F = \arg \max_{\ve x \in \cP} \ve c \cdot \ve x$. Every face
$F$ of $\cP$ is also a polytope. If $\cP$ is of dimension $D$, a face $F$ of dimension $D{-}1$
is a called a \emph{facet}. For $k \in \N$, we define the
$k$-th dilation of $P$ as $k \cP := \left\{\, k \ve x \mid \ve x \in \cP 
\right\}$. A point $\ve x \in \cP$ is a \emph{vertex} if it can
not be written as a convex combination of points from $\cP \backslash \{\ve
x\}$.

The \emph{cone} of $\{\ve a_1,\ldots,\ve a_m\} \subset \R^d$ is defined as 
\begin{equation*}   
\cone (\ve a_1,\ldots,\ve a_m) := \left\{\, \ve x \in \R^d \mid \ve x = \sum_{i=1}^m \lambda_i \ve a_i, \, \lambda_i \geq 0 \, \right\}.  
\end{equation*} 
Thus, $\cone(A)$ denote the cone over the columns of a matrix $A$. 
We are interested in the case when the matrix in consideration is the 
design matrix $A^T$.
We define the \emph{design polytope}  $\cP^{T}$ as  the convex hull
$\conv(A^T)$ and we write $C^T$ to denote $\cone(A^T)$.

Given an integer matrix $A\in\Z^{d\times m}$ we associate an integer lattice
$\Z A = \{n_1 \ve a_1 +\cdots + n_m \ve a_m \mid n_i\in \Z\}$. We can also associate
the semigroup $\N A := \{n_1 \ve a_1 + \cdots + n_m \ve a_m \mid n_i \in \N\}$. We say
that the semigroup $\N A$ is \emph{normal} when $\ve x \in \N A$ if and only if
there exist $\ve y\in \Z^m$ and $\alpha\in \R_{\geq 0}^m$ such that $\ve x = A
\ve y$ and $\ve x = A \alpha$. 
The set $\Z A \cap \cone(A)$
is called the \emph{saturation} of $\N A$. See
\cite{MS2005,Sturmfels1996} for more details on normality. 

If $\ve x \in \R^{6}$, we index $\ve x$ by $\{\, ij \mid 1 \leq i , j
\leq 3,\, i \neq j\, \}$. We define $\ve e_{ij} \in \R^{6}$ to be the vector
of all zeros, except $1$ at index $ij$.  We also adopt the notation $x_{i+} :=
\sum_{j} x_{ij}$ and $x_{+i} := \sum_{j} x_{ji}$. For any $\ve x \in
\N^{6}$ we can define a directed multigraph $G(\ve x)$ on three vertices,
where there are $x_{ij}$ directed edges from vertex $i$ to vertex $j$. One
would like to identify the vectors $\ve x\in \N^{6}$ for which the graph
$G(\ve x)$ is a state graph. Nevertheless, observe that $x_{i+}$ is the
out-degree of vertex $i$ and $x_{+i}$ is the in-degree of vertex $i$ with
respect to $G(\ve x)$.

We now give some properties which will be used later for describing the facets of the
design polytope $\cP^{T}$ given by the design matrix for our model, and to prove normality of
the semigroup associated with the design matrix.

\begin{proposition}[Proposition 5.1 in \cite{Haws:2011fk}]\label{prop:bnddeg}
Let $A^T$ be the design matrix for the THMC without loops and initial parameters.
If $\ve x \in \Z A^T \cap C^T$ then $\sum_{i \neq j} x_{ij} = n(T{-}1)$ for
some $n \in \N$ and $|x_{i+} - x_{+i}| \leq n$ for all $i \in \{1,2,3\}$.
\end{proposition}

An immediate consequence of this proposition is that $\cP^T \subset \R^6$ has dimension 5. Proposition \ref{prop:bnddeg} also states that for $\ve x \in \Z A^T \cap
C^T$ the multigraph $G(\ve x)$ will have in-degree and out-degree bounded
by $\|\ve x\|_1 / (T{-}1)$ at every vertex. This implies nice properties when $\|\ve x\|_1 = (T{-}1)$. 
Recall that a path in a directed multigraph is
\emph{Eulerian} if it visits every edge only once.

\begin{proposition}[Proposition 5.2 in \cite{Haws:2011fk}] \label{prop:eulpath}
If $G$ is a directed multigraph on three vertices, with no self-loops, $T{-}1$ edges, and satisfying
\begin{equation*}
|G_{i+} - G_{+i}| \leq 1 \qquad \, i=1,2,3;
\end{equation*}
then, there exists an Eulerian path in $G$.
\end{proposition}

Note that every word $\ve w \in \MS(\wordsnlt)$ gives an Eulerian path
in $G(\{\ve w\})$
containing all edges.  Conversely, for every multigraph $G$ with an Eulerian
path containing all edges, there exists $\ve w \in \MS(\wordsnlt)$ such that
$G(\{\ve w$\})$ = G$. More specifically, $\ve w$ is the Eulerian path
in $G(\{\ve w\})$.
Throughout this paper we use the terms \emph{path} and \emph{word} interchangeably.

\begin{lem}[Lemma 5.2 in \cite{Haws:2011fk}] \label{lem:integinterior}
Let $A^T$ be the design matrix for the THMC.
If  $T \geq 4$, then $\cP^T \cap \Z^{6} = A^T$, where
the right hand side is taken as the set of columns of the matrix $A^T$.
\end{lem}


We define
\begin{equation*}
H_{n(T{-}1)} := \left\{ \ve x \in \R^{6} \mid \sum_{i \neq j} x_{ij} = n(T{-}1) \right\}.
\end{equation*}

\begin{proposition}[Proposition 5.3 in \cite{Haws:2011fk}] \label{prop:conehype}
Let $A^T$ be the design matrix for the THMC without initial
parameters and no loops.
\begin{enumerate}
\item For $T \geq 4$ and $n \in \N$,
\begin{equation*}
n\cP^T = C^T \cap H_{n(T-1)}.
\end{equation*}
\item For $T \geq 4$,
\begin{equation*}
C^T \cap \Z A^T = \bigoplus_{n=0}^\infty \left( n\cP^T \cap \Z^6\right).
\end{equation*}
\end{enumerate}
\end{proposition}




\subsection{Facets of $\cP^T$ for $T\ge 5$}

In this section we describe the facets of the design polytope $\cP^{T}$
for arbitrary $T\geq 5$. For $T=3,4$ it can
be easily checked that $\cP^{T}$ has 12 facets using {\tt Polymake} \cite{polymake}.  For $T\ge 5$ the output of 
{\tt Polymake} suggests that $\cP^T$ always has 24 facets.
In the following we establish this fact by explicitly describing all these facets for all $T\geq 5$.

Recall that vectors $\ve c \in \R^6$ are
indexed as $[c_{12}, c_{13},c_{21},c_{23}, c_{31},c_{32}]$. 
The proofs for the facets of $\cP^{T}$ rely heavily on the state graph.
Note that we give the facets of the design polytope in terms of equivalence
classes under permutations of the labels $\{1,2,3\}$. For example, if $\mathfrak{S}_3$ denotes the set of permutations of the set $\{1,2,3\}$ and if the vector
$[c_{12},c_{13},c_{21}, c_{23}, c_{31}, c_{32}]$ defines a facet, then for any
$\sigma\in \mathfrak{S}_3$, the vector
$$
[c_{\sigma(1)\sigma(2)},c_{\sigma(1)\sigma(3)},c_{\sigma(2)\sigma(1)}, c_{\sigma(2)\sigma(3)}, c_{\sigma(3)\sigma(1)}, c_{\sigma(3)\sigma(2)}]
$$ 
also defines a facet.
Also note that, due to Proposition~\ref{prop:conehype}, the facets of the polytope $\cP^T$ and the cone $C^T$ are not only in bijection but they can be determined by the same linear inequalities of the form $\ve c \cdot \ve x\geq 0$. We call $\ve c$ the vector defining 
a facet of $\cP^T$ or $C^T$.

Recall that in general, to show that a vector $\ve c$ defines a facet of $C^T$, we need to show
the following two things:
\begin{itemize}
\item[i)] (Non-negativity) $\ve c \cdot \ve a^{T}_{\ve w}$ are non-negative for
  all $\ve w \in \wordsnlt$.
\item[ii)] (Dimensionality) the dimension of linear subspace spanned by $\{
  \ve a^T_{\ve w} \mid \ve c \cdot \ve a^T_{\ve w}=0\}$ is $5$.
\end{itemize}

\begin{proposition}
For any $T\ge 5$ 
\[
\ve c=[1,0,0,0,0,0]  
\]
defines a facet of $\cP^{T}$ modulo $\mathfrak{S}_3$.
\end{proposition}

\begin{proof}
The non-negativity i) follows by definition, as the transition count between states should
be a non-negative integer. Now, for dimensionality ii), consider the following four paths
\begin{equation}\label{eq:4paths}
2313\cdots 13, \ 
1323\cdots 23, \ 
3232\cdots 32, \ 
3131\cdots 31.
\end{equation}
These paths have sufficient statistics that depend on $T$.
The vectors of sufficient statistics for $T=2k$ are
\begin{align*}
[0,k{-}1,0,1,k{-}1,0], \; 
[0,1,0,k{-}1,0,k{-}1], \;
[0,0,0,k{-}1,0,k], \;
[0, k{-}1, 0,0,k,0 ];
\end{align*}
and when $T=2k+1$ the sufficient statistics are
\begin{align*}
[0,k{-}1,0,1,k,0], \; 
[0,1,0,k{-}1,0,k], \;
[0,0,0,k,0,k], \;
[0, k, 0,0,k,0 ].
\end{align*}
Thus, the four paths in~\eqref{eq:4paths} correspond to four vectors that are linearly independent for $T\ge 5$. 
Lastly, consider the following path
$$
132132\cdots 132.
$$
Its vector of transition counts contains a nonzero value in the coordinate corresponding to the transition $21$, which shows that the space generated by the transition counts of these five paths is 5-dimensional. We conclude 
by noticing that none of these paths contain the transition $12$, thus they satisfy $\ve c \cdot \ve a^T_{\ve w}=0$.
\end{proof}


\begin{proposition}
\label{prop:1}
For any $T\ge 5$ 
\[
\ve c=[T,\, T,\, -(T{-}2),\, 1,\, -(T{-}2),\, 1]  
\]
defines a facet of $\cP^{T}$ modulo $\mathfrak{S}_3$.
\end{proposition}

\begin{proof}
We check the non-negativity i).  Consider any particular word (path) $\ve w \in \wordsnlt$ of length $T$ with transition counts
$x_{12},x_{13},x_{21},x_{23},x_{31},x_{32}$.  We need to show
\begin{equation}
\label{eq:1}
T(x_{12}+x_{13}) + x_{23}+x_{32} \ge (T{-}2) (x_{21}+x_{31}).
\end{equation}
Note that $x_{12}+x_{13}$ is the out-degree of vertex 1 and $x_{21}+x_{31}$ is
the in-degree of vertex 1 with respect to the graph $G(\ve w)$. By Proposition
\ref{prop:bnddeg} the out-degree and the in-degree can differ by at most $1$.
Note that \eqref{eq:1} trivially holds when $x_{12}+x_{13} \ge x_{21}+x_{31}$.
Hence we only need to
check the case $a=x_{12}+x_{13} = x_{21}+x_{31}-1$.
Now 
\[
T{-}1=x_{12}+x_{13}+x_{21}+x_{23}+x_{31}+x_{32}=2a+1+x_{23}+x_{32}
\]
or
\[
2a+2 + x_{23}+x_{32} - T=0.
\]
Hence the difference of two sides of \eqref{eq:1} is written as
\begin{equation}
\label{eq:2}
Ta + x_{23}+x_{32} - (T{-}2)(a+1)=x_{23}+x_{32} +2(a{+}1)-T
=0.
\end{equation}
This proves the non-negativity.

Next we consider dimensionality ii). Equation \eqref{eq:1} can not hold with equality in the case
$x_{12} + x_{13} > x_{21} + x_{31}$. Also \eqref{eq:1} can not hold with equality
in the case  $x_{12} + x_{13} = x_{21} + x_{31}>0$.  Furthermore, if $0=x_{12}+x_{13}=
x_{21}+x_{31}$, then the path entirely consists of edges between 2 and 3.
Then $T{-}1=x_{32}+x_{23} > 0$ and \eqref{eq:1} does not hold with equality.
Hence the only remaining case is $x_{12} + x_{13} = x_{21} + x_{31} -1$.
But then, from \eqref{eq:2} we see that \eqref{eq:1} holds with equality.
Therefore, all paths $\ve w$ such that $x_{12} + x_{13} = x_{21} + x_{31} -1$ satisfies
$\ve c\cdot \ve a_{\ve w}^T=0$.
We now give five such paths with linearly independent sufficient
statistics, which depends on $T \bmod 3$ and $T \bmod 2$.

If $T$ even, consider
\begin{equation*}
3131\cdots 131,\; 2121\cdots 121,\; 3232\cdots 3231 
\end{equation*}

If $T$ odd, consider
\begin{equation*}
23131\cdots 31,\; 32121\cdots 21,\; 2323\cdots 231
\end{equation*}

If $T\equiv 0 \  \pmod{3}$, consider
\begin{equation*}
321321\cdots 321,\; 231231\cdots231
\end{equation*}

If $T\equiv 2 \pmod{3}$, consider
\begin{equation*}
213213\cdots 2131,\; 312312\cdots3121
\end{equation*}

Finally if $T\equiv 1 \pmod{3}$, put  the loop $232$ or $323$  in front of 
the word above for the value of $T\equiv 2  \pmod{3}$.

We need to show that the sufficient statistics of these paths are linearly independent.  For example, 
consider the case $T=6k$.  Then the sufficient statistics are given  by the vectors
\begin{align*}
[0,3k{-}1,0,0,3k,0],\; [3k{-}1,0,3k,0,0,0],\; [0,0,0,3k{-}1,1,3k{-}1],\\ 
[0, 2k{-}1, 2k, 0, 0, 2k],\; [0,2k,0,2k{-}1,2k,0].
\end{align*}
For $k\ge 1$, the linear independence of these five vectors can be easily verified.
Other cases $T \equiv r \pmod{6}$ can be similarly handled. 
\end{proof}


In a similar fashion, we prove the following propositions.

\begin{proposition}
\label{prop:2}
For any $T=2k+1$ odd, $T\ge 5$,  
\[
\ve c=[1,\, 1,\, -1,\, -1,\, 1,\, 1]
\]
defines a facet of $\cP^{T}$ modulo $\mathfrak{S}_3$.
\end{proposition}

\begin{proof}
For non-negativity, consider
\begin{equation}
\label{eq:3}
x_{12}+x_{13} + x_{31} + x_{32} \ge x_{21}+x_{23}.
\end{equation}
We can ``merge'' two vertices 1 and 3 as a virtual vertex 4 and consider
4 as a single vertex. Then the resulting graph has only two vertices (2 and 4).
Then $x_{12}+x_{32}$ is the out-degree of this vertex 4.
If $x_{12}+x_{32} \ge x_{21}+x_{23}$ then the inequality is trivial.
Consider the case $x_{12}+x_{32} = x_{21}+x_{23}-1$.  We need to show that
in this case we have $x_{13} + x_{31} \ge 1$.  
By contradiction assume that $x_{13}+x_{31}=0$. Then the path of odd length is like
\[
2424242.
\]
However in this case $x_{12}+x_{32} = x_{21}+x_{23}$, which is a contradiction.
Therefore we have proved that
\eqref{eq:3} holds for any path $\ve w$.

Now we check the dimensionality ii).  In order to check the dimensionality 
we verify for which path \eqref{eq:3} holds with equality.
The first case is $0=x_{13}+x_{31}$.  Then as we saw above we have
$x_{12}+x_{32} = x_{21}+x_{23}$.  Other case is $1=x_{13}+x_{31}$, i.e.\ 
either $x_{13}=1, x_{31}=0$ or $x_{31}=1, x_{13}=0$. In the former case
the path is like
\[
2121323
\]
and in the latter case the path is like $2323121$.
We claim that for $T=2k+1 \geq 5$ the paths
\begin{equation*}
121\cdots 121,\; 232\cdots 232,\; 212 \cdots 2123,\; 1323 \cdots 232,\; 3121\cdots 212
\end{equation*}
give sufficient statistics that are linearly independent and hold with equality for Equation \eqref{eq:3}.
The sufficient statistics for the paths above are
\begin{align*}
[k,0,k,0,0,0],\; [0,0,0,k,0,k],\; [k{-}1,1,k,0,0,0],\\
 [0, 1, 0, k{-}1, 0, k],\; [k,0,k{-}1,0,1,0].
\end{align*}
For $k\ge 2$, one can check the linear independence of these 5 vectors.
\end{proof}

We now consider any three consecutive transitions, or a path for $T=4$.  
Let $\tilde x_{ij}$, $1\le i\neq j \le 3$, be transition counts of
these three transitions.  
\begin{lem}
\label{lem:2} 
Let $i,j,t$ be distinct (i.e.\  $\{i,j,t\}=\{1,2,3\}$).
Then
\[
\tilde x_{ij} + \tilde x_{jt} + \tilde x_{it} \ge 1.
\]

\end{lem}
\begin{proof}
Suppose that $\tilde x_{ij} + \tilde x_{jt} + \tilde x_{it}=0$.  Then
in the three transitions, we can not use the directed edges
$ij$, $jt$, $it$.  Then the available edges are $ji,ti,tj$.  By drawing a state graph,
it is obvious that by the edges $ji, ti,tj$ only, we can not form a path of three
transitions.
\end{proof}

\begin{proposition}
\label{prop:3}
For any $T\ge 7$ of the form $T=3k{+}1$ ($k\ge 2$), 
\[
\ve c=[2,\,-1,\,-1,\,-1,\, 2,\, 2]
\]
defines a facet of $\cP^{T}$ modulo $\mathfrak{S}_3$.
\end{proposition}
\begin{proof}
For non-negativity, consider the inequality
\begin{equation}
\label{eq:4}
2(x_{12} + x_{31} + x_{32}) \ge x_{13} + x_{21} + x_{23}
\end{equation}
Since $x_{13} + x_{21} + x_{23} = T{-}1 - (x_{12} + x_{31} + x_{32})
=3k-(x_{12} + x_{31} + x_{32})$,
the expression~\eqref{eq:4} is equivalent to 
\begin{equation}
\label{eq:5}
x_{12} + x_{31} + x_{32} \ge k .
\end{equation}
If we consider paths in triples of transitions, 
\eqref{eq:5} follows from Lemma \ref{lem:2}.

Now for checking the dimensionality
we consider the case when the inequality \eqref{eq:5} becomes an equality.   By the induction above, if  we divide a path into triples of
transitions (edges), then in each triple only one of $\tilde x_{12}, \, \tilde x_{31},\, \tilde x_{32}$ has to be 1.
That is, we proved above that for every three transitions (edges) the left-hand-side
of equation \eqref{eq:5} increases by one. For equality to hold, the LHS can
only increase by exactly one.
Knowing this, we consider the cases for which three transitions increases
the LHS of Equation \eqref{eq:5} by exactly one.
In three transitions, a path can either come back to the same vertex or
move to another vertex.  In the former case (say $ijti$), 
the following three loops $1321$,$3213$,$2132$ increases $x_{32}$ by 1.
Another case going from $i$ to $j$ in three transitions are of the form
\[
ijij, ijtj, itij,
\]
where $i,j,t$ are different. Then appropriate ones are
only the following ones:
\[
2121,1313,2323.
\]
Therefore in three transitions, we go $2\rightarrow 1$, $1\rightarrow 3$ or
$2\rightarrow 3$. Among these three, the only possible connection is
\[
2\rightarrow 1\rightarrow 3
\]
(or $2\rightarrow 3$ alone).  Then the loops are inserted at any point. Thus, we can consider the following paths 
\begin{multline}\label{pathsforeq5}
\qquad 2121321321\cdots 1321, \, 
1313213\cdots 3213, \,
2323213\cdots 3213, \\
2132132\cdots 2132,\,
2121313\cdots 1313,\qquad 
\end{multline}
with sufficient statistics given by the vectors
\begin{align*}
[1,k{-}1,k{+}1,0,0,k{-}1], \,
[0,k{+}1,k{-}1,0,1,k{-}1], \, 
[0,k{-}1,k{-}1,2,0,k], \\
[0,k,k,0,0,k],\,
[1,2(k{-}1), 2, 0, k{-}1, 0];
\end{align*}
respectively.  For $k\ge 2$ it is easily checked that
these vectors are linearly independent and satisfy Equation \eqref{eq:5} with equality. 
\end{proof} 

\begin{proposition}
\label{prop:4}
For any $T\ge 5$ of the form $T=3k{+}2$, $k\ge 1$, 
\[
\ve c=[2k{+}1,\, -k,\, -k,\, -k,\, 2k{+}1,\, 2k{+}1]
\]
defines a facet of $\cP^{T}$ modulo $\mathfrak{S}_3$.
\end{proposition}
\begin{proof}
For non-negativity, consider
\begin{equation}
\label{eq:4_4}
(2k{+}1)(x_{12} + x_{31} + x_{32}) \ge k(x_{13} + x_{21} + x_{23}) .
\end{equation}
Since $x_{13} + x_{21} + x_{23} = T{-}1 - (x_{12} + x_{31} + x_{32})
=3k{+}1-(x_{12} + x_{31} + x_{32})$,
the inequality \eqref{eq:4_4} is equivalent to 
\begin{equation*}
\label{eq:4_5}
(2k{+}1)(x_{12} + x_{31} + x_{32}) \ge k(3k{+}1 - (x_{12} + x_{31} + x_{32})),
\end{equation*}
which simplifies to 
\begin{equation}
\label{eq:4_6}
x_{12} + x_{31} + x_{32} \ge k.
\end{equation}

For a path of length $3k{+}2$, consider omitting
the last transition.  Then we have a path of length $3k{+}1$.
The inequality already holds for this shortened path
by Proposition \ref{prop:3}. Since
the last transition only increases the transition counts, the
same inequality holds for $3k{+}2$.  This proves the non-negativity.

For dimensionality, we can find five paths by adding one of the transitions $2\rightarrow 1$, $1\rightarrow 3$, or $2\rightarrow 3$ either at the end or at the beginning of the paths 
in Proposition \ref{prop:1}. 
In this way, we obtain the following paths
\begin{align*}
2121321321\cdots 1321\underline{3}, \, 
2321321321\cdots 1321\underline{3}, \, 
\underline{2}1313213\cdots 3213, \\
\underline{1}3213\cdots3213,\,
2132132\cdots 2132\underline{1}, \,
\end{align*}
with the following vectors of frequencies:
\begin{align*}
 [1,k,k{+}1,0,0,k{-}1], \,
[0,k,k{+}1,1,0,k{-}1], \,
[0,k{+}1,k,0,1,k{-}1], \\
[0,k{+}1,k,0,0,k], \,
[0,k,k{+}1, 0,0,k];
\end{align*}
which are easily checked to be linearly independent and satisfy \eqref{eq:4_6} with
equality.

\end{proof}

The following lemma will be useful to show the facets of $\cP^T$ when $T\geq 6$ is even.
\begin{lem}\label{lemForProp5}
Let $T=2k$, with $k\ge 1$. Then, the inequality $2x_{12} + x_{13} +x_{32} \ge k{-}1$ holds for every path. Moreover, the inequality is strict for every path ending at $s_T = 2$. 

\end{lem}

\begin{proof}
We prove the lemma by induction on $k$.  

For $k=1$, this is a path with a single transition; thus, the statement is obvious,
because $2x_{12} + x_{13} +x_{32}$ is non-negative and for two paths $12$, $32$ ending
at $s_2=2$, we have $2x_{12} + x_{13} +x_{32}=1 \ \text{or}\ 2$.

Assume that the proposition holds for  $k$.  Now we prove that it holds for $k{+}1$.
For a path $\ve w$ let 
\begin{eqnarray*}
\ve x&=&(x_{12},x_{13},x_{21},x_{23},x_{31},x_{32})\\
&=& (x_{12}(\ve w),x_{13}(\ve w),x_{21}(\ve w),x_{23}(\ve w),x_{31}(\ve
w),x_{32}(\ve w))
\end{eqnarray*}
denote the transition counts of $\ve w$.
Consider a path $\ve w$ of length $T=2k{+}2$
\[
\ve w=s_1 \dots s_{2k} s_{2k+1} s_{2k+2}.
\]
Denote $\ve w^0=s_1 \dots s_{2k}$ and 
$\ve w^1=s_{2k} s_{2k+1} s_{2k+2}$. Then 
\begin{multline*}
\qquad 2x_{12}(\ve w) + x_{13}(\ve w) +x_{32}(\ve w) = (2x_{12}(\ve w^0) + x_{13}(\ve w^0) +x_{32}(\ve w^0))  \\
+ (2x_{12}(\ve w^1) + x_{13}(\ve w^1) +x_{32}(\ve w^1)).\qquad
\end{multline*}
The inductive assumption is that 
\[
2x_{12}(\ve w^0) + x_{13}(\ve w^0) +x_{32}(\ve w^0) \ge k{-}1
\]
and
\[
s_{2k}=2 \ \Rightarrow  \ 2x_{12}(\ve w^0) + x_{13}(\ve w^0) +x_{32}(\ve w^0) \ge k .
\]

We  prove the first statement of the proposition. If
\[
2x_{12}(\ve w^1) + x_{13}(\ve w^1) +x_{32}(\ve w^1)\ge 1,
\]
then the inequality holds for $\w$. On the other hand it is easily seen that 
if $2x_{12}(\ve w^1) + x_{13}(\ve w^1) +x_{32}(\ve w^1)=0$ then the only possible case is 
$\ve w^1=231$.  Then  we have
$s_{2k}=2$.  Hence by the second part of the inductive assumption we also have
the inequality.

We now prove the second statement of the proposition. Let $s_{2k+2}=2$. 
Note that $s_{2k+1}$ is either 1 or 3.  If  $s_{2k+1}=1$, then
\[
2x_{12}(\ve w^1) + x_{13}(\ve w^1) +x_{32}(\ve w^1) =2
\]
and the inequality for $\ve w$ is  strict.  On the other hand let 
$s_{2k+1}=3$, then there are two cases:
\[
\ve w^1=232 \ \ \text{or}\ =132 .
\]
In the former case $2x_{12}(\ve w^1) + x_{13}(\ve w^1) +x_{32}(\ve w^1) =1$, but
$s_{2k}=2$. Hence by the inductive assumption the inequality is strict.  In the latter case
$2x_{12}(\ve w^1) + x_{13}(\ve w^1) +x_{32}(\ve w^1) =2$ and the
inequality is strict.
\end{proof}

\begin{proposition}
\label{prop:5}
For any even $T\ge 6$ 
\[
\ve c = [\frac{3}{2}T{-}1, \, \frac{T}{2}, \, -\frac{T}{2}{+}1, \,  -\frac{T}{2}{+}1,\,  -\frac{T}{2}{+}1, \, \frac{T}{2}]
\]
defines a facet of $\cP^{T}$ modulo $\mathfrak{S}_3$.
\end{proposition}
\begin{proof} Write $T=2k$, $k\ge 3$.
For non-negativity, consider
\[
(3k{-}1)x_{12} + k(x_{13} + x_{32}) \ge (k{-}1)(x_{21} + x_{23} + x_{31}) .
\]
Substituting
\[
x_{21} + x_{23} + x_{31} = T{-}1 - (x_{12} + x_{13} + x_{32}) = 2k{-}1 - (x_{12} + x_{13} + x_{32})
\]
into the above and collecting terms, we have
\[
(4k{-}2)x_{12} + (2k{-}1)(x_{13} + x_{32}) \ge (k{-}1)(2k{-}1)
\]
or equivalently 
\begin{equation}
\label{eq:6}
2x_{12} + x_{13} + x_{32} \ge k{-}1,
\end{equation}
which holds by Lemma \ref{lemForProp5}

For dimensionality, consider the following 5 paths.
\[
31\ldots 31, \ 32\ldots3231, \ 2131\ldots31, \ 2313\ldots13, \ 23123131\ldots31
\]
The sufficient statistics for these  paths are
\begin{multline*}
\qquad [0,k{-}1,0,0,k,0],\; 
[0,0,0,k{-}1,1,k{-}1],\; 
[0,k{-}1,1,0,k{-}1,0],\\
 [0,k{-}1,0,1,k{-}1,0],\; 
 [1,k{-}3,0,2,k{-}1,0].\qquad
\end{multline*}
For $k\ge 3$, linear independence of these vectors can be easily checked. 
\end{proof}

We now state some results that will be useful to treat the remaining case when $T=3l$.
\begin{lem}
\label{lem:1}
Suppose that $\ve u_1, \dots,\ve u_l\in {\mathbb R}^m$ are linearly independent vectors
such that $[1,1,\dots,1]\cdot  \ve u_j=c>0$ is a positive constant for $j=1,\dots,l$.
Then for any non-negative vector $\ve w\in {\mathbb R}_{\geq 0}^m$, 
$\ve u_1 + \ve w, \dots, \ve u_l+\ve w$ are linearly independent.
\end{lem}

\begin{proof}
For scalars $\alpha_1, \dots, \alpha_l$ consider
\begin{equation}
\label{eq:11}
0 = (\ve u_1 {+} \ve w) \alpha_1 + \dots + (\ve u_l {+}\ve w)\alpha_l
  = \ve u_1 \alpha_1 + \dots + \ve u_l \alpha_l + \ve w (\alpha_1+ \dots +\alpha_l)
\end{equation}
Taking the inner product with $[1,1,\dots,1]$ we have
\[
0 = (c + [1,1,\dots,1]\cdot {\ve w}) (\alpha_1 + \dots+ \alpha_l).
\]
Here $c+[1,1,\dots,1]\cdot {\ve w} > 0$.  Hence we have $0=\alpha_1 + \dots + \alpha_l$.
But then \eqref{eq:11} reduces to
\[
0=\ve u_1 \alpha_1 + \dots + \ve u_l \alpha_l
\]
By linear independence of $\ve u_1,\dots,\ve u_l$ we have
$\alpha_j=0$, $j=1,\dots,l$.
\end{proof}


\begin{lem}
\label{lem:3} 
Let $\tilde x_{ij}$, $1\le i\neq j \le 3$, be transition counts of
three consecutive transitions and let 
$i,j,t$ be distinct.
Then
\begin{equation}
\label{eq:3steps-1}
2   \tilde x_{ij} +  \tilde x_{it} +   \tilde x_{tj} \ge   \tilde x_{ji}.
\end{equation}
The equality holds for the following paths: $jiji$, $jtji$, $jiti$. These three
paths start at $j$ and end at $i$.
Furthermore, if the difference of both sides is 1 then
the possible transitions in three steps are 
$j\rightarrow t$, $t\rightarrow i$, and self-loops
$i\rightarrow i$, $j\rightarrow j$, $t\rightarrow t$.
Finally, if the difference of both sides is 2, then
the possible transitions in three steps are
$t\rightarrow j$,  and $i\rightarrow t$, and self-loops
$i\rightarrow i$, $j\rightarrow j$, $t\rightarrow t$.
\end{lem}
\begin{proof}
If $ \tilde x_{ji}\le 1$, then the inequality is obvious from Lemma \ref{lem:2}.
If $ \tilde x_{ji}=2$, then the only possible path is $jiji$, for which
the equality holds in 
\eqref{eq:3steps-1}.

We now consider the values of the difference of both sides.
First we determine paths, where the equality holds.
$jiji$ is the unique solution for $ \tilde x_{jt}=2$.  Consider
the case $ \tilde x_{ji}=1$. Then the equality holds only
if $ \tilde x_{ij}=0$ and one of $ \tilde x_{it}$ and $ \tilde x_{tj}$ is 1.
It is easy to check that the former case corresponds only to $jiti$ and
the latter case corresponds only to $jtji$.

We now enumerate the cases that the difference is 1.
If $ \tilde x_{ji}=0$, then $ \tilde x_{ij}=0$ and one of
$ \tilde x_{it}$ or $ \tilde x_{tj}$ is zero. This is only 
possible for the paths $jtit$ or $tjti$.\ 
(Recall that $ \tilde x_{ji}=2$ leads to $jiji$, for which the 
difference is zero, as treated above.)\ 
Now consider $ \tilde x_{ji}=1$. The case $ \tilde x_{ij}=1$
corresponds to $jijt$ or $tiji$.
The case $ \tilde x_{ij}=0$ corresponds the loop
$jitj$, $itji$, or $tjit$. We now see that
the transitions in three steps are 
$j\rightarrow t$,   $t\rightarrow i$,
or the self-loops $i\rightarrow i$, $j\rightarrow j$, $t\rightarrow t$.

Finally we enumerate the cases that the difference is 2.
It is easy to see that $ \tilde x_{ji}\le 1$.
First suppose $ \tilde x_{ji}=0$.  If $ \tilde x_{ij}=1,  \tilde x_{it}= \tilde x_{tj}=0$,
then this corresponds to loops (in reverse direction than in the previous 
case) $ijti$, $jtij$, $tijk$, resulting in self-loops in three steps.
If $ \tilde x_{ij}=0,  \tilde x_{it}= \tilde x_{tj}=1$, then 
this corresponds to $titj$ or $itjt$.
If  $ \tilde x_{ij}=0,  \tilde x_{it}=2,  \tilde x_{tj}=0$, then 
this corresponds to $itit$.  Similarly 
$ \tilde x_{ij}=0,  \tilde x_{it}=0,  \tilde x_{tj}=2$ corresponds to
$tjtj$.
Second suppose $ \tilde x_{ji}=1$. LHS has to be 3.
Then $ \tilde x_{ij}=1$ and one of $ \tilde x_{it}$ and $ \tilde x_{tj}$ is 1.
It is easy to see that these correspond to paths
$tjij$ and $ijit$.  Then we see that in three steps, the possible
transitions are $t\rightarrow j$, $i\rightarrow t$ or
the three self-loops.
\end{proof}

Using Lemma \ref{lem:3} 
we now consider 6 consecutive
transitions (in two triples).  Let 
$\bar x_{ij}$, $1\le i\neq j \le 3$, denote transition counts of
6 consecutive transitions.  Then we have the following lemma.

\begin{lem}
\label{lem:3a} 
Let $i,j,t$ be distinct.
Then
\begin{equation}
\label{eq:6steps}
2 \bar x_{ij} + \bar x_{it} + \bar x_{tj} \ge \bar x_{ji} + 1.
\end{equation}
When the equality holds, then the path has to start from $j$ and
end at $i$ in 6 steps.
When the difference of both sides is 1, then the possible
transitions in 6 steps are $j\rightarrow i$, $j\rightarrow t$ and $t\rightarrow i$.
\end{lem}
\begin{proof}
From the previous lemma, 
\[
2 \bar x_{ij} + \bar x_{it} + \bar x_{tj} \ge \bar x_{ji} 
\]
but the equality is impossible, because then the path would have to go from state $j$ to state $i$
in three steps twice.  Hence 
\eqref{eq:6steps} holds.

Now consider the case of equality. Then differences of two triples are 0 and 1.
By the previous lemma, the order of 0 before 1 only corresponds to 
$j\rightarrow i \rightarrow i$.   The order of 1 before 0 corresponds
to  $j\rightarrow j \rightarrow i$.  Hence in both cases, the paths
have to start from $j$ and end at $i$ in 6 steps.

Now consider the case of difference of 1, i.e.
\[
2 \bar x_{ij} + \bar x_{it} + \bar x_{tj} = \bar x_{ji}  + 2 .
\]
The two differences of two triples are $(0,2)$, $(1,1)$ or $(2,0)$.
In the case of $(0,2)$, by the previous lemma, the transitions in 6 steps
are $j\rightarrow i$ , $j\rightarrow t$. 
In the case of $(1,1)$, the transitions
in 6 steps are $j\rightarrow t$, $t\rightarrow i$ or $j\rightarrow i$. 
In the case of $(2,0)$, the transitions are $t\rightarrow i$ or
$j\rightarrow i$.  In summary, the possible transitions in 6 steps
are $j\rightarrow i$, $j\rightarrow t$, $t\rightarrow i$.

\end{proof}

The final lemma is as follows.
\begin{lem}
\label{lem:5}
Consider a path of length $T=6k{+}1$, i.e., path with $6k$ steps.
Then
\begin{equation}
\label{eq:lem5}
2 x_{ij} + x_{it} + x_{tj} \ge x_{ji}  + 2k{-}1.
\end{equation}
If the equality holds, then the path has to be at $i$ at time $T$.
\end{lem}
\begin{proof}
We divide a path into $k$ subpaths of length 6.
Suppose that there exists a block for which
\begin{equation}
\label{eq:deficit}
2 \bar x_{ij} + \bar x_{it} + \bar x_{tj} = \bar x_{ji}  +1.
\end{equation}
In this block the path goes from $j$ to $i$.  Before another
block of this type, the path has to come back to $j$.  But then
there has to be some block of $i\rightarrow t$ or $i\rightarrow j$.
For these blocks
\begin{equation}
\label{eq:profit}
2 \bar x_{ij} + \bar x_{it} + \bar x_{tj} \ge  \bar x_{ji}  +3.
\end{equation}
Therefore, the deficit of 1 in \eqref{eq:deficit} is compensated
by the gain of 1 in \eqref{eq:profit}.  The lemma follows from this
observation.  The condition for equality also follows from this observation.
\end{proof}

\bigskip

Now we will show a facet for the cases $T=6k$ and $T=6k+3$, $k=1,2,\dots$.

\begin{proposition}
\label{prop:11}
For $T=6k{+}3$  \ (with $k\ge 1$)
\[
\ve c = [ 5k{+}2,\, 2k{+}1,\, -4k{-}1,\, -k,\, -k,\, 2k{+}1]
\]
defines a facet of $\cP^{T}$ modulo $\mathfrak{S}_3$.
\end{proposition}

\begin{proof}
For non-negativity, consider
\begin{equation}
\label{eq:22}
(5k{+}2)x_{12} + (2k{+}1)x_{13} + (2k{+}1) x_{32} \ge (4k{+}1) x_{21} + k x_{23} + k x_{31}.
\end{equation}
The RHS is written as
\[
(3k{+}1) x_{21} + k( x_{21} + x_{23} + x_{31})=
(3k{+}1) x_{21} + k (6k{+}2 - x_{12} - x_{13} - x_{32}).
\]
Hence \eqref{eq:22} is equivalent to
\[
(6k{+}2) x_{12} + (3k{+}1) x_{13} + (3k{+}1) x_{32} \ge (3k{+}1) x_{21} +  2k (3k{+}1).
\]
Dividing by $3k{+}1>0$, this is equivalent to
\begin{equation}
\label{eq:6k3}
2 x_{12} + x_{13} + x_{32} \ge  x_{21} +  2k .
\end{equation}
A path has $T{-}1=6k{+}2$ steps.
Consider the first $6k$ steps divided into triples of steps and apply  
Lemma \ref{lem:5} to the LHS.  
If the inequality in 
\eqref{eq:lem5} is strict, then we only need to check
that 
\[
2 x_{12} + x_{13} + x_{32} \ge  x_{21}
\]
for the remaining two steps.  This is obvious, because a transition
$21$ has to be preceded or followed by some term on the left-hand side.

If equality holds in \eqref{eq:lem5}, at time $6k{+}1$,
the path is at state $1$ and then the penultimate step
is either $12$ or $13$.  In either case we easily see that 
\[
2 x_{12} + x_{13} + x_{32} >  x_{21}.
\]
This prove the non-negativity.

For dimensionality, consider the following $5$ paths.
\begin{align*}
23132132132132\ldots 132132,\,
 321321\ldots 321,\,
213232321321321\ldots21,  \\
213213213\ldots213,\,
21321321321\ldots3213212121.
\end{align*}
The sufficient statistics for these paths are
\begin{align*}
[0,2k, 2k, 1, 1, 2k],\,
[0, 2k, 2k{+}1, 0, 0, 2k{+}1],\,
[0,2k{-}1,2k,2,0,2k{+}1]\\
[0,2k{+}1,2k{+}1,0,0,2k],\,
[2,2k{-}1,2k{+}2,0,0,2k{-}1].\\
\end{align*}
These are linearly independent for $k\ge 1$.
\end{proof}

\bigskip
Now we consider $T=6k$.
\begin{proposition}
\label{prop:22}
For $T=6k$ \ (with $k\ge 1$)
\[
\ve c = [10k{-}1,\, 4k,\, -8k{+}2,\, -2k{+}1,\, -2k{+}1,\, 4k]
\]
defines a facet of $\cP^{T}$ modulo $\mathfrak{S}_3$.
\end{proposition}

\begin{proof} 
For non-negativity, consider 
\begin{equation}
\label{eq:T6k}
(10k{-}1) x_{12} + 4k x_{13} + 4k x_{32} \ge (8k{-}2) x_{21} + (2k{-}1) x_{23} + (2k{-}1) x_{31}.
\end{equation}
The RHS is written as
\[
(6k{-}13) x_{21} + (2k{-}1)(6k{-}1 - x_{12} - x_{13} - x_{32})
\]
Hence \eqref{eq:T6k} is equivalent to
\[
(12k{-}2) x_{12} + (6k{-}1) x_{13} + (6k{-}1) x_{32} \ge (6k{-}1) x_{21} + (2k{-}1)(6k{-}1)
\]
or 
\[
2 x_{12} +  x_{13} + x_{32} \ge x_{21} + (2k{-}1).
\]
The rest of the proof is similar to that of Proposition \ref{prop:11}.

For dimensionality, consider the following $5$ paths.
\begin{align*}
232321321321\ldots 321,\,
213231321321\ldots 321,\,
321321321321\ldots 321,\, \\
213213\ldots213, \,
212321321321\ldots321.
\end{align*}
The sufficient statistics for these paths are
\begin{align*}
[0,2k{-}2, 2k{-}1, 2, 0, 2k],\,
[0, 2k{-}1, 2k{-}1, 1, 1, 2k{-}1],\,
[0,2k{-}1,2k,0,0,2k]\\
[0,2k,2k,0,0,2k{-}1],\,
[1,2k{-}2,2k,1,0,2k{-}1].\\
\end{align*}
These are linearly independent for $k\ge 1$.
\end{proof}

\subsubsection{Summary of facets}

Here we summarize all the inequalities that define the facets of the cone $C^T$ thus, of the polytope $P^T$ as well, for all $T\ge 5$. 
We only present one of the six vectors defining the facets, with the understanding that any permutation of the labels $\{1,2,3\}$ leads to another facet inequality. 

\begin{table}[htbp]
\caption{Vectors defining facets of $C^T$ (homogeneous inequalities)}
\label{tab:homogeneous}
\begin{center}
\begin{tabular}{|c|c|}\hline
$T$ & $\ve c$ \  in $\ve c \cdot \ve x \geq 0$\\  \hline
all & $[1,0,0,0,0,0]$\\
all &  $[T,T,-(T-2),1,-(T-2),1]$ \\
odd &  $[1,1,-1,-1,1,1]$ \\
even & $[\frac{3}{2}T{-}1, \frac{T}{2}, -\frac{T}{2}+1, -\frac{T}{2}+1, -\frac{T}{2}+1, \frac{T}{2}]$\\
$3k+1$&$[2,-1,-1,-1,2,2]$\\
$3k+2$&$[2k+1,-k,-k,-k,2k+1,2k+1]$\\ 
$6k+3$&$[5k+2, 2k+1, -4k-1, -k, -k, 2k+1]$\\
$6k$&$[10k-1, 4k, -8k+2, -2k+1, -2k+1, 4k]$\\ \hline
\end{tabular}
\end{center}
\end{table}


Notice that the last four inequalities are listed according to the value of $T \mod 3$. 

In this list of facets, some of the vectors $\ve c$ depend on $T$.
However, this vector $\ve c$ defines also a facet inequality for any dilation $n\cP^T$ of the design polytope by Proposition~\ref{prop:conehype}. 
By substituting the equality
$n(T{-}1) = x_{12} + x_{13} + x_{21} + x_{23} + x_{31} + x_{32}$ into the original inequality defined by $\ve c$, we obtain an inequality for the dilated polytope $n\cP^T$ where $\ve c$ does not depend on $T$. For instance, the inequality 
for the second $\ve c=[T,T,-(T{-}2),1,-(T{-}2),1]$ in Table \ref{tab:homogeneous} we have
\begin{eqnarray}
\nonumber &[T,\, T,\, -(T{-}2),\, 1,\, -(T{-}2),\, 1] \cdot \ve x \geq 0\\
\Longleftrightarrow & T(x_{12} +x_{13}) + (x_{23} + x_{32}) \geq
(T{-}2)(x_{21}+x_{31}).\label{inhom}
\end{eqnarray}
Substituting $x_{23} + x_{32} = n(T{-}1) -
(x_{12}+x_{13}+x_{21}+x_{31})$ in the inequality \eqref{inhom}, we
get 
\[
\begin{array}{cc}
&T(x_{12} +x_{13}) + n(T{-}1) -
(x_{12}+x_{13}+x_{21}+x_{31}) \geq
(T{-}2)(x_{21}+x_{31})\\
\Longleftrightarrow & (T{-}1)(x_{12} +x_{13}) - (T{-}1) (x_{21}+x_{31})
\geq -n(T{-}1) \\
\Longleftrightarrow & (x_{12} +x_{13}) -  (x_{21}+x_{31})
\geq -n\\
\Longleftrightarrow & \tilde {\ve c} \cdot \ve x \geq -n,\\
\end{array},
\]
where 
\[
\tilde {\ve c}= [1,1,-1,0,-1,0].
\]
Notice that the last inequality defines a linear form with a nonzero constant term which defines the same supporting hyperplane for $n\cP^T$ as the inequality~\eqref{inhom}.  Also notice that $\tilde {\ve c}$
is proportional to the main order term of $\ve c=[T,T,-(T{-}2),1,-(T{-}2),1]$ as $T\rightarrow\infty$.

We refer to the inequalities listed in Table \ref{tab:homogeneous}
as \emph{homogeneous}, because they define inequalities for the facets of $C^T$ (thus for those of $n\cP^T$) and we call \emph{inhomogenous} to those inequalities for $n\cP^T$ derived by substituting the equality $n(T{-}1) = x_{12} + x_{13} + x_{21} + x_{23} + x_{31} + x_{32}$.
Inhomogeneous inequalities are essential in Section~\ref{normal} for the proof of normality of the semigroup associated with the
design matrix $A^{T}$. Inhomogenous inequalities are of the form
\[
\tilde {\ve c} \cdot \ve x \ge n \, a
\]
for some $n,\, a\in \N$; where $\tilde {\ve c}$ and $a$ depend only on $T \ {\rm mod}\  6$.
In Table \ref{tab:inhomogeneous}, we summarize the inhomogeneous inequalities corresponding with those 
from Table~\ref{tab:homogeneous} above.

\begin{table}[htbp]
\caption{Vectors defining facets of $nP^T$ (inhomogenous inequalities)}
\label{tab:inhomogeneous}
\begin{center}
\begin{tabular}{|c|c|c|}\hline
$T$ & $\tilde {\ve c}$ \  in $\tilde {\ve c} \cdot \ve x \ge n \, a$ & $a$\\ \hline
all & $[1,0,0,0,0,0]$ & $0$\\
all & $[1,1,-1,0,-1,0]$& $-1$\\  
odd & $[1,1,-1,-1,1,1]$ & $0$ \\  
even& $[3,1,-1,-1,-1,1]$ & $-1$ \\
$3k+1$ & $[2,-1,-1,-1,2,2]$ & $0$ \\
$3k+2$ & $[2,-1,-1,-1,2,2]$ & $-1$ \\
$6k+3$ & $[5,2,-4,-1,-1,2]$ & $-2$\\
$6k$ & $[5,2,-4,-1,-1,2]$ & $-2$ \\ \hline
\end{tabular}
\end{center}
\end{table}

\subsection{There are only 24 facets for $T\geq 5$}

In the previous section, we gave 24 facets of the polytope $\cP^{T}$ for every $T\geq 5$. 
Here, we discuss how these 24 facets are enough to describe the polytope $P^{T}$ 
(the convex hull of the columns of $A^{T}$), depending on $T$.


Recall that the columns of $A^T$ are on the following hyperplane
\[
H_{T-1}=\{ (x_{12},\dots,x_{32}) \mid T{-}1=x_{12} + \dots + x_{32}\}.
\]
Then it is clear by Proposition \ref{prop:conehype} that
\[
P^{T} = C^{T} \cap H_{T-1}.
\]
Let ${\mathcal F}_T$ denote the set of 
linear forms of the supporting hyperplanes for the pointed cone $ C^{T}$.
Then the linear forms of the supporting hyperplanes $F$ of $P^{T}$
(within $H_{T-1}$) are of the form $F \subseteq H_{T-1}, F \subseteq {\mathcal
  F}_T$. 


For every $T$, let $\tilde {\mathcal F}_T$ denote the 24 facets prescribed in the previous section, and let ${\mathcal F}_T$ denote the set of all facets of $P^{T}$.
Therefore we have a certain subset $\tilde {\mathcal F}_T\subset {\mathcal F}_T$  and 
we need to show that $\tilde {\mathcal F}_T={\mathcal F}_T$.
Let $\tilde {\mathcal C}_T$ denote the polyhedral cone defined by $\tilde {\mathcal F}_T$.
It follows that $\tilde {\mathcal C}_T \supset {\mathcal C}_T$.
Note that $\tilde {\mathcal F}_T = {\mathcal F}_T$ if and only
if $\tilde {\mathcal C}_T = {\mathcal C}_T$. Also let
\[
\tilde P_T = \tilde {\mathcal C}_T \cap H_{T-1} .
\]
Then $\tilde P_T \supset P^{T}$ and
$\tilde P_T = P^{T}$ if and only if 
$\tilde {\mathcal C}_T = {\mathcal C}_T$.  

The above argument shows that to prove 
$\tilde {\mathcal F}_T = {\mathcal F}_T$ it suffices to show
that
\begin{equation}
\label{eq:24-1}
\tilde P_T \subset P^{T}.
\end{equation}
Let $\tilde V_T$ be the set of vertices of $\tilde P_T$. Then
in order to show 
\eqref{eq:24-1}, it suffices to show that
\[
\tilde V_T\subset P^{T}.
\]
Hence, if we can obtain explicit expressions of the vertices of 
$\tilde V_T$ and can show that each vertex belongs to $P^{T}$, we are done.

In the previous section, we used only the condition $T{-}1=x_{12} + \dots + x_{32}$ to settle the equivalence between the homogeneous and inhomogeneous inequalities defining the 24 the facets of $P^{T}$.
Hence 
the homogeneous and the inhomogeneous inequalities are equivalent
on $H_{T-1}$.  Therefore, for each $r=0,\dots,5$, there exists a polyhedral region
defined by 24 fixed affine half-spaces $\{ \ve x\mid \tilde {\ve c} \cdot \ve x \ge a\}$ of Table \ref{tab:inhomogeneous} (for $n=1$), say $Q^r$,  such that
\[
\tilde {\mathcal P}_T =  Q^r \cap H_{T-1}, \quad   T=6k+r, \, k=1,2,\dots
\]

Since  $Q^r$ is a polyhedral region it  can be 
written as a Minkowski sum of a polytope $P^r$ and a cone $C^r$:
\begin{equation}
\label{eq:Minkowski-sum}
Q^r = P^r + C^r.
\end{equation}
Please note that  $r$  is modulo 6,  but $T$ is not.
Recall the Minkowski sum of two sets $A, B \subseteq \R^d$ is simply $\{\, a + b \mid a \in A,\, b \in B \, \}$.
The six cones and polytopes defining $Q^r$ for $r=0,\ldots,5$ were computed using {\tt Polymake} \cite{polymake} and they are given in the Appendix.
For each  vertex $\ve v$ of
$P^r$ and each  extreme ray $\ve e$ of $C^r$ let $l_{\ve v,\ve e}$ denote
the half-line emanating 
from $\ve v$ in the direction $\ve e$:
\[
l_{\ve v,\ve e}=\{ \ve v + t\ve e \mid  t\ge 0\}
\]
Given the explicit expressions of $v$ and $e$ we can solve
\[
[1,1,1,1,1,1]\cdot  (\ve v+t\ve e)=T{-}1
\]
for $t$ and get
\[
t:=t(T,\ve v,\ve e)=\frac{T{-}1-[1,\dots,1]\ve v}{(1,\dots,1)\ve e}.
\]
Then $v+t(T,\ve v,\ve e)\ve e \in H_{T-1}$.
Note that
\[
\tilde V_T \subset \{\ve v+t(T,\ve v,\ve e)\ve e \mid \ve v:\text{vertex of $P^r$}, 
\ \ve e:\text{extreme ray of $C^r$}\}.
\]
Also clearly
\[
\{v+t(T,\ve v,\ve e)\ve e \mid \ve v:\text{vertex of $P^r$}, 
\ \ve e:\text{extreme ray of $C^r$}\} \subset \tilde P_T={\rm conv}(\tilde V_T).
\]

The above argument shows that for proving $\tilde{\mathcal F}_T = {\mathcal F}_T$
it suffices to show that 
\begin{equation}
\label{eq:24-2}
\{\ve v+t(T,\ve v,\ve e)\ve e \mid \ve v:\text{vertex of $P^r$}, 
\ve e:\text{extreme ray of $C^r$}\} \subset P^{T}.
\end{equation}

For proving \eqref{eq:24-2} the following lemma is useful.
\begin{lem}
\label{lem:rays}
Let $\ve v \in P^r$ and $\ve e \in C^r$. If $\ve v + t(T,\ve v,\ve e)\ve e \in P^{T} \cap \Z^6$ for some $T$,
then $\ve v + t(T + 6k, \ve v,\ve e)\ve e \in P^{T + 6k}$ for all $k \geq 0$.
\end{lem}
\begin{proof}
If $\ve x := \ve v + t(T,\ve v,\ve e)\ve e \in P^{T} \cap \Z^{6}$ for some $T$ then $\ve x$ corresponds
to a path of length $T$ on three states with no loops (word in $\wordsnlt$). Suppose $\ve e$ is a two-loop (three-loop) e.g.\ 121 (1231).
Then $\ve x + (3k)\ve e \in P^{T + 6k}$ ($\ve x + (2k)\ve e \in P^{T + 6k}$). That is, since $\ve x$ is an integer
point (a path) contained in $P^{T}$, we can simply add three (or two depending on the loop) copies of
the loop $\ve e$ and we will be guaranteed to have a path of the correct length meaning it will
be contained in $P^{T+6}$.
\end{proof}

By this lemma we need to compute $C^r$ only for some special small $T$'s.  
We computed all vertices and all rays for the cases $T=12,7,20,9,16,11$. 
The software to generate the design matrices can be found at 
\url{https://github.com/dchaws/GenWordsTrans}
and  the design matrices and some other material can be found at
\url{http://www.davidhaws.org/THMC.html}.  By our computational result and Lemma 
\ref{lem:rays} we verified the following proposition.

\begin{proposition}
\label{prop:explicit-rays}
The rays of the cones $C^r$ for $r=0,\ldots,5$ are 
$[1,0,1,0,0,0]$, $[1,0,0,1,1,0]$, $[0,1,1,0,0,1]$, $[0,1,0,0,1,0]$, $[0,0,0,1,0,1]$. In terms of
the state graph, the rays correspond to the five loops 121, 131, 232, 1231, and 1321.
\end{proposition}

Note that $C^r$, $r=0,\ldots,5$ are common and we denote them as $C$ hereafter.
Also note that the rays of the cone $C^r$ are very  simple.
Proposition \ref{prop:explicit-rays} implies the following theorem.
\begin{thm} \label{thm:facets}
The 24 facets given in Propositions \ref{prop:1}, \ref{prop:2}, \ref{prop:3}, \ref{prop:4}, \ref{prop:5}, \ref{prop:11}, \ref{prop:22} (depending on $T \mod 6$) are all the facets
of $P^{T} = \conv(A^{T})$ for $T\ge 5$.
\end{thm}



\section{Normality of the semigroup}\label{normal}

From the definition of normality of a semigroup in Section \ref{facets}, the semigroup $\N A^T$ defined by the design matrix is normal if it coincides with the elements in both, the integer lattice $\Z A^T$ and the cone $C^T$. 

In this section, we provide an inductive proof of the normality of the semigroup $\N A^T$ for arbitrary $T\ge 3$. We verified the normality for the first cases by computer. 

\begin{lem}\label{normal1}
The semigroup $\N A^{T}$ is normal for $3 \leq T \leq
135$. 
\end{lem}

\begin{proof}
The normality of the design matrices $A^{T}$ for $3 \leq T \leq 135$ was
confirmed computationally using the software {\tt Normaliz} \cite{normaliz}.
The software to generate the design matrices and the scripts to run the
computations are available at \url{https://github.com/dchaws/GenWordsTrans}.
\end{proof}


Using Lemma \ref{normal1} as a base, we prove normality in the general case by induction.
\begin{thm}\label{thm:normality}
The semigroup $\N A^{T}$ is normal for any $T \ge 3$.
\end{thm}
\begin{proof}
We need to show
that given any transition counts $x_{12},\dots,x_{32}$, such that
their sum is divisible by $T{-}1$ and  the counts lie in ${\rm cone}({A^{T}})$,
there exists a set of paths
having these transition counts.
Write $\ve x=[x_{12},x_{13},x_{21},x_{23},x_{31},x_{32}]^T$ and 
$1_6=[1,1,1,1,1,1]$. Let
\[
n=1_6 \cdot \ve x/(T{-}1)
\]
denote the number of the paths.  
Note that $n$ is determined from $T$ and $\ve x$.

We listed above inhomogeneous forms of inequalities
defining facets.  
In all cases $T=6k{+}r$ for $r=0,1,\dots,5$, the inhomogeneous inequalities 
for $n$ paths  can be put in the form
\begin{equation}
\label{eq:inhom}
c_{12}  x_{12} + \dots + c_{32} x_{32} \ge a(n{+}1),
\end{equation}
where $c_{12},\dots, c_{32}, a\in \N$ do not depend on $n$.
Since $C=C_r$ are common for $r=0,\dots,5$, 
the expression~\eqref{eq:Minkowski-sum} is written as
\[
Q^r = P^r + C.
\]
The $n$-th dilation of $Q^r$ is 
\[
Q_n^r:= n Q^r = n P^r + nC = n P^r + C.
\]
Then from \eqref{eq:inhom} we have
\[
{\rm cone}({A^{T}}) \cap \{ \ve x \mid 1_6 \cdot \ve x = n(T{-}1)\} = 
Q_n^r \cap \{ \ve x \mid 1_6 \cdot \ve x = n(T{-}1)\}.
\]

We now look at vertices of $P^r$ from Appendix.
The vertex [0,3,4,3,0,7] for $Q^2$ has the largest $L_1$-norm, which
is 17. Hence the sum of elements
of these vertices $P^r_n$ are at most $17n$.


\medskip
For $T=6k{+}r$, any 
non-negative integer vector $\ve x\in {\rm cone}({A^{T}})$ such that $1_6 \cdot \ve x=n(T{-}1)$, can be written
as
\[
\ve x= \ve b + \alpha_1 \ve e_1 + \alpha_2 \ve e_2 + \alpha_3 \ve e_3 + \alpha_4 \ve e_4 + \alpha_5 \ve e_5, \quad \ve b\in nP^r_1, \  \alpha_i \ge 0, \ i=1,\dots,5.
\]
Taking the inner product with $1_6$ (i.e.\ the $L_1$-norm) we have 
\begin{equation}
\label{eq:6n}
1_6 \cdot  \ve b \le 17n.
\end{equation}
Hence
\[
n(T{-}1)=1_6 \cdot \ve  x \le 17n + 2 (\alpha_1 + \alpha_2 + \alpha_3) + 3 (\alpha_4 + \alpha_5).
\]

Consider the case that
\begin{equation}
\label{eq:alphas}
\alpha_1, \alpha_2, \alpha_3 \le 3n, \quad \alpha_4, \alpha_5 \le 2n.
\end{equation}
Then
\[
n(T{-}1)=1_6 \cdot \ve x  \le 17n + 18n + 12n = 47n
\]
or $T\le 48$.  Hence if $T > 48$ 
we have at least one of
\begin{equation}
\label{eq:alphas1}
\alpha_1> 3n,\  \alpha_2 > 3n,\  \alpha_3> 3n, \ \ \alpha_4 >2n,\ \alpha_5 > 2n.
\end{equation}

\medskip
Now we employ induction on $k$ for  $T=6k{+}r$.  Note that
we arbitrarily fix $n \ge 1$ and use induction on $k$.
For $k\le 21$ we have $T=6k{+}r \le 126 + r \le 131 < 135$ 
and the normality holds by
the computational results.

Now consider $k > 22$ and let $T=6k{+}r$.
In this case at least one inequality of \eqref{eq:alphas1} holds.
Let 
\[
\ve x\in  {\rm cone}({A^{T}}) \cap \{ x \mid 1_6 \cdot \ve x = n(T{-}1)\} 
\]
First consider $\ve x$ such that 
$\alpha_1 > 3n$.  (The argument for $\alpha_2$ and $\alpha_3$ is the same.)
Let
\[
\tilde {\ve x} = \ve x - 3n \ve e_1 \in  {\rm cone}({A^{T}}) \cap \{ \ve x
\mid 1_6 \cdot \ve  x = n(T{-}1{-}6)\} 
\]
Our inductive assumption is that there exists a set of paths $\ve w_1, \dots, \ve w_n$  of length $T{-}6$
having $\tilde x$ as the transition counts.
We now form $n$ partial paths of length $6$:
\[
n \ \text{times}\ ijijij
\]
Note that instead of $ijijij$ we can also use $jijiji$.
We now argue that these $n$ partial paths can be appended (at the end or at the beginning)
of each path $\ve w_1, \dots, \ve w_n$.  

Let $\ve w_1=s_1 \dots s_{T-6}$.  If $s_1 \neq s_{T-6}$, then
\[
\{s_1, s_{T-6}\}\cap \{i,j\} \neq \emptyset,
\]
since $|S|=3$.  In this case we see that at least one of the following 4 operations
is possible
\begin{enumerate}
\setlength{\itemsep}{0pt}
\item put $ijijij$ at the end of $\ve w_1$
\item put $jijiji$ at the end of $\ve w_1$
\item put $ijijij$ in front of $\ve w_1$
\item put $jijiji$ in front of $\ve w_1$
\end{enumerate}
Hence $\ve w_1$ can be extended to a path of length $T$.
Now consider the case that  $s_1= s_{T-6}$, i.e., $\ve w_1$ is a cycle.  It may happen
that $s_1 \neq i, j$.  But a cycle can be rotated, i.e., instead of $\ve w_1=s_1 \dots s_{T-6}$ we can take
\[
\ve w_1'=s_2 s_3 \dots s_{T-6} s_1
\]
where $s_2 \neq s_1$, hence $s_2=i \ \text{or}\ j$.  Then either
$ijijij$ or $jijiji$ can be put in front of $\ve w_1'$ and $\ve w_1'$ can be extended.  Therefore
we see that $\ve w_1$ can be extended in any case.
Similarly $\ve w_2,\dots,\ve w_n$ can be extended.

The case of $\alpha_4 >2n$ is trivial.  The path $ijkijk$ can be rotated as
$jkijki$ or $kijkij$.  Therefore one of them can be appended to each of $\ve w_1, \dots, \ve w_n$.
\end{proof}

\section{Discussion}\label{discussion}

In this paper, we considered only the situation of the toric
homogeneous Markov chain (THMC) model \eqref{thmc} for $S = 3$, with the extra assumption of having non-zero transition probabilities only when the transition is between two different states.
In this setting, we described the hyperplane representations of the
design polytope for any $T \geq 3$, 
and from this representation we showed that the semigroup generated by the columns
of the design matrix $A^{T}$ is normal.  

We recall from Lemma 4.14 in \cite{Sturmfels1996}, that a given set of integer vectors $\{\ve a_1,\ldots, \ve a_m\}$ is a graded set if there exists
$\ve w \in \Q^{ S^2}$ such that ${\bf a}_i \cdot \ve w = 1$. In our setting, the set of columns of the design matrix $A^{T}$ is a graded set, as each of its columns add up to $T{-}1$, so we let $\ve w = (\frac{1}{T{-}1},\ldots,\frac{1}{T{-}1})$.

In his book\cite{Sturmfels1996}, Sturmfels provided a way to bound the generators of the toric ideal associated to an integer matrix $A$. The precise statement is the following.
\begin{thm}[Theorem 13.14 in\cite{Sturmfels1996}]\label{normal_thm}
Let $A \subset \Z^d$ be a graded set such that the semigroup generated
by the elements in $A$ is normal.  Then the toric ideal $I_A$ associate with
the set $A$ is generated by homogeneous binomials of degree at most $d$.
\end{thm}

In particular, the normality of the semigroup  generated by the columns of the design matrix $A^{T}$ is demonstrated in Theorem~\ref{thm:normality}; therefore, we obtain the following theorem as a consequence of Theorem~\ref{normal_thm}.

\begin{thm}\label{thm:bound}
For $S=3$ and for any $T\geq 3$, a Markov basis for the toric ideal $I_{A^T}$ associated to the THMC model (without loops and initial parameters) consists of binomials of degree at most $6$.
\end{thm}

The bound provided by Theorem \ref{thm:bound} seems not to be sharp, in the sense that there exists Markov basis whose elements have degree strictly less than $6$. In our computational experiments, we found evidence that more should be true.  

\begin{conj}\label{conj:MBbound}
Fix $S= 3$; then, for every $T\geq 3$, there is a Markov basis for the toric
ideal $I_{A^{T}}$ consisting of binomials of degree at most $2$, and there is
a Gr\"obner basis with respect to some term ordering consisting of binomials of
degree at most $3$.
\end{conj}

In general,  we do not know the degree of a
Markov basis for the toric 
ideal $I_{A^{T}}$ nor the smallest degree of a Gr\"obner
basis for fixed $S\geq 4$.  However, for $S = 4$  we  
found that the degree of a Markov basis (using {\tt 4ti2}~\cite{4ti2}) is $4$ for $T = 3, \, 4$ and the degree is $3$ for $T = 5$.  Unfortunately, {\tt 4ti2}  was not able to compute a Markov basis
for $T \geq 6$.
We also noted that for $S \geq 4$, the semigroup generated by the columns
 of $A^{T}$ is not normal.  For example, for $S = 4$ and  $T = 8$, the 
 linear combination $\frac{1}{2} {\bf a}^{4,
   4}_{12121212} + \frac{1}{2} {\bf a}^{4,
   4}_{34343434}$ is an integral solution in the intersection between the
cone and the integer lattice.  However, this does not form a path.  
Thus, for any $S \geq 4$ and any $T \geq 5$, it is interesting to investigate 
the necessary and sufficient conditions that impose normality for the semigroup 
 generated by the columns of the
design matrix $A^{T}$.

\section{Appendix}
The six cones used in defining $Q^r$ for $r=0,\ldots,5$:
\begin{eqnarray*}C^r &:=& \cone \big(\, [1,0,1,0,0,0], [1,0,0,1,1,0], [0,1,1,0,0,1], \\
&&\qquad [0,1,0,0,1,0], [0,0,0,1,0,1] \,\big)
\end{eqnarray*}
for $r=0,\ldots,5$.

The six polytopes used in defining $Q^r$ for $r=0,\ldots,5$ are given below,
where the vertices are modulo the permutations of $S = \{1,2,3\}$. That is, the
indexing below is $x_{12}, x_{21}, x_{13}, x_{31}, x_{23}$, and $x_{32}$. To get
the full list of vertices one should use all six permutations of $\{1,2,3\}$
and permute the indices of each vertex below accordingly.

$\vertices( Q^0) := \Big[$ 
[0,    1,    1,    0,    0,    1],
[0,    1,    2,  1/2,    1,  3/2],
[0,    1,  3/2,    1,  1/2,    2],
[0,    2,    2,    0,    1,    2],
[0,    2,    2,    0,    2,    5],
[0,    2,    2,  2/3,    0,  7/3],
[0,    2,    3,    0,    2,    4],
[0,    2,    4,    2,    0,    3],
[0,    2,  3/2,    0,    0,  3/2],
[0,    2,  7/3,    0,  2/3,    2],
[0,    3,    4,    0,    0,    4],
[0,  6/5,  8/5,  4/5,  2/5, 11/5],
[0,  6/5, 11/5,  2/5,  4/5,  8/5],
[2/3,  4/3,  4/3,  2/3,  2/3,  7/3]$\Big]$.

$\vertices( Q^1) := \Big[$ 
[0,    0,    0,    0,    0,    0],
[0,    0,    0,    1,    3,    2],
[0,    0,    1,    0,    2,    3],
[0,    1,    1,    0,    1,    3],
[0,    1,    1,    0,    2,    2],
[0,    1,    2,    0,    1,    2],
[0,    1,    2,    1,    0,    2],
[0,    1,  1/2,  1/2,  1/2,  1/2],
[0,  1/2,    0,  1/2,    1,    1],
[0,  1/2,    1,    1,  1/2,    0],
[0,  1/2,  1/2,    1,  1/2,  1/2],
[0,  1/2,  1/2,  1/2,    1,  1/2] $\Big]$.

$\vertices( Q^2) := \Big[$ 
[0,    1,    1,    0,    1,    2],
[0,    1,    2,  1/2,    2,  5/2],
[0,    1,    3,  3/2,    1,  3/2],
[0,    1,  3/2,    1,  3/2,    3],
[0,    1,  5/2,    2,  1/2,    2],
[0,    2,    2,    0,    2,    5],
[0,    2,    2,    0,    3,    4],
[0,    2,    2,    1,    2,    4],
[0,    2,    3,    0,    2,    4],
[0,    2,    4,    2,    0,    3],
[0,    2,    4,    2,    1,    2],
[0,    3,    4,    0,    3,    7],
[0,    3,    7,    3,    0,    4],
[1/3,  2/3,  2/3,  1/3,  1/3,  2/3],
[1/3,  2/3,  5/3,  5/6,  4/3,  7/6],
[1/3,  2/3,  7/6,  4/3,  5/6,  5/3],
[2/3,  4/3,  4/3,  2/3,  2/3,  7/3],
[2/3,  4/3,  4/3,  2/3,  5/3,  4/3] $\Big]$.

$\vertices( Q^3) := \Big[$ 
[0,    0,    0,    0,    0,    0],
[0,    0,    0,    1,    1,    0],
[0,    1,    1,    0,    2,    4],
[0,    1,    2,    0,    2,    3],
[0,    1,    3,    2,    0,    2] $\Big]$.

$\vertices( Q^4) := \Big[$ 
[0,    1,    1,    0,    0,    1],
[0,    1,    2,  1/2,    1,  3/2],
[0,    1,  3/2,    1,  1/2,    2],
[0,    2,    2,    0,    1,    4],
[0,    2,    2,    0,    2,    3],
[0,    2,    2,    1,    1,    3],
[0,    2,    3,    0,    1,    3],
[0,    2,    3,    1,    0,    3],
[0,    2,    3,    1,    1,    2],
[0,    3,    4,    0,    2,    6],
[0,    3,    6,    2,    0,    4],
[1/3,  5/3,  5/3,  1/3,  1/3,  8/3],
[1/3,  5/3,  5/3,  1/3,  4/3,  5/3] $\Big]$.

$\vertices( Q^5) := \Big[$ 
[0,    0,    0,    0,    1,    1],
[0,    0,    0,    1,    4,    3],
[0,    0,    1,    0,    3,    4],
[0,    1,    1,    0,    2,    4],
[0,    1,    1,    0,    3,    3],
[0,    1,    2,    0,    2,    3],
[0,    1,    3,    2,    0,    2],
[0,    1,  1/2,  1/2,  3/2,  3/2],
[0,    1,  3/2,  3/2,  1/2,  1/2],
[0,  1/2,    0,  1/2,    2,    2],
[0,  1/2,    2,    2,  1/2,    0],
[0,  1/2,  1/2,  1/2,    2,  3/2],
[0,  1/2,  3/2,    2,  1/2,  1/2],
[1,    2,  1/2,  1/2,  1/2,  1/2] $\Big]$.

\nocite{Hara:2010vn}
\nocite{Hara:2010uq}
\bibliographystyle{spmpsci}      
\bibliography{references}

\begin{thebibliography}{10}
\providecommand{\url}[1]{{#1}}
\providecommand{\urlprefix}{URL }
\expandafter\ifx\csname urlstyle\endcsname\relax
  \providecommand{\doi}[1]{DOI~\discretionary{}{}{}#1}\else
  \providecommand{\doi}{DOI~\discretionary{}{}{}\begingroup
  \urlstyle{rm}\Url}\fi

\bibitem{normaliz}
Bruns, W., Ichim, B., S\"oger, C.: {\tt Normaliz}, a tool for computations in
  affine monoids, vector configurations, lattice polytopes, and rational cones
  (2011)

\bibitem{diaconis-sturmfels}
Diaconis, P., Sturmfels, B.: Algebraic algorithms for sampling from conditional
  distributions.
\newblock The Annals of Statistics \textbf{26}(1), 363--397 (1998)

\bibitem{polymake}
Gawrilow, E., Joswig, M.: polymake: a framework for analyzing convex polytopes.
\newblock In: G.~Kalai, G.M. Ziegler (eds.) Polytopes --- Combinatorics and
  Computation, pp. 43--74. Birkh\"auser (2000)

\bibitem{Hara:2010uq}
Hara, H., Takemura, A.: A markov basis for two-state toric homogeneous markov
  chain model without initial paramaters.
\newblock Journal of Japan Statistical Society \textbf{41}, 33--49 (2011)

\bibitem{Haws:2011fk}
Haws, D., {Martin del Campo}, A., Yoshida, R.: Degree bounds for a minimal
  markov basis for the three-state toric homogeneous markov chain model.
\newblock Proceedings of the Second CREST--SBM International Conference,
  ``Harmony of Grobner Bases and the Modern Industrial Society'' pp. 99 -- 116
  (2012)

\bibitem{MS2005}
Miller, E., Sturmfels, B.: Combinatorial commutative algebra.
\newblock Graduate texts in mathematics. Springer (2005).
\newblock \urlprefix\url{http://books.google.com/books?id=CqEHpxbKgv8C}

\bibitem{Pachter:2005kx}
Pachter, L., Sturmfels, B.: Algebraic Statistics for Computational Biology.
\newblock Cambridge University Press, Cambridge, UK (2005)

\bibitem{stochastics}
Pardoux, E.: Markov Processes and Applications: Algorithms, Networks, Genome
  and Finance.
\newblock Wiley Series in Probability and Statistics Series. Wiley, John \&
  Sons, Incorporated (2009)

\bibitem{Schrijver1986Theory-of-linea}
Schrijver, A.: Theory of Linear and Integer Programming.
\newblock John Wiley \& Sons, Inc., New York, NY, USA (1986)

\bibitem{Sturmfels1996}
Sturmfels, B.: Gr\"obner Bases and Convex Polytopes, \emph{University Lecture
  Series}, vol.~8.
\newblock American Mathematical Society, Providence, RI (1996)

\bibitem{Hara:2010vn}
Takemura, A., Hara, H.: Markov chain monte carlo test of toric homogeneous
  markov chains.
\newblock Statistical Methodology. doi:10.1016/j.stamet.2011.10.004.
  \textbf{9}, 392--406 (2012)

\bibitem{4ti2}
4ti2 team: 4ti2---a software package for algebraic, geometric and combinatorial
  problems on linear spaces.
\newblock {A}vailable at www.4ti2.de

\end{thebibliography}

\end{document}